\documentclass[12pt]{article}

\usepackage{amssymb}
\usepackage{amscd}
\usepackage{fullpage}
\usepackage{graphics,amsmath,amssymb}
\usepackage{amsthm}
\usepackage{amsfonts}
\usepackage{float}

\title{Counting the Palstars}

\author{L. Bruce Richmond\\
Combinatorics and Optimization\\
University of Waterloo \\
Waterloo, ON  N2L 3G1 \\
Canada \\
{\tt lbrichmo@uwaterloo.ca} \\
\ \\
Jeffrey Shallit\\
School of Computer Science\\
University of Waterloo \\
Waterloo, ON  N2L 3G1 \\
Canada \\
{\tt shallit@cs.uwaterloo.ca} }

\begin{document}

\maketitle

\theoremstyle{plain}
\newtheorem{theorem}{Theorem}
\newtheorem{corollary}[theorem]{Corollary}
\newtheorem{lemma}[theorem]{Lemma}
\newtheorem{proposition}[theorem]{Proposition}

\theoremstyle{definition}
\newtheorem{definition}[theorem]{Definition}
\newtheorem{example}[theorem]{Example}
\newtheorem{conjecture}[theorem]{Conjecture}

\theoremstyle{remark}
\newtheorem{remark}[theorem]{Remark}

\begin{abstract}
A palstar (after Knuth, Morris, and Pratt) is a concatenation of
even-length palindromes.
We show that, asymptotically, there are $\Theta(\alpha_k^n)$ palstars of
length $2n$ over a $k$-letter alphabet, where $\alpha_k$ is a constant
such that $2k-1 < \alpha_k < 2k-{1 \over 2}$.  In particular, 
$\alpha_2 \doteq 3.33513193$.
\end{abstract}

\section{Introduction}
We are concerned with finite strings over a finite alphabet $\Sigma_k$
having $k \geq 2$ letters.   A {\it palindrome} is a string $x$
equal to its reversal $x^R$, like the English word {\tt radar}.  
If $T, U$ are sets of strings over $\Sigma_k$ then (as usual)
$TU = \lbrace tu \ : \ t \in T, u \in U \rbrace$.  Also
$T^i = \overbrace{TT \cdots T}^i$ and 
$T^* = \bigcup_{i \geq 0} T^i$ and $T^+ = \bigcup_{i \geq 1} T^i$.

We define 
$$P = \lbrace x\, x^R \ : \ x \in \Sigma_k^+ \rbrace,$$
the language of nonempty even-length palindromes.  Following
Knuth, Morris, and Pratt \cite{Knuth&Morris&Pratt:1977}, we call a string $x$
a {\it palstar} if it belongs to 
$P^*$, that is, if it can be written as the concatenation of elements of $P$.
Clearly every palstar is of even length.

We call $x$ a 
{\it prime palstar} if it is a nonempty palstar, but not the concatenation 
of two or more palstars; alternatively, if $x \in P^+ - P^2 P^* $ where
$-$ is set difference.
Thus, for example, the 
the English word {\tt noon} is a prime palstar, but the English
word {\tt appall} and the 
French word {\tt assailli} are palstars that are not prime.
Knuth, Morris, and Pratt \cite{Knuth&Morris&Pratt:1977}
proved that no prime palstar is a proper prefix
of another prime palstar, and, consequently, every palstar has a unique
factorization as a concatenation of prime palstars.

A nonempty string $x$ is a {\it border} of a string $y$ if
$x$ is both a prefix and a suffix of $y$ and $x \not= y$.  
We say a string $y$ is {\it bordered} if it has a border.
Thus, for example, the English word {\tt ionization}
is bordered with border {\tt ion}.  Otherwise a word is {\it unbordered}.
Rampersad et al.~\cite{Rampersad&Shallit&Wang:2011}
recently gave a bijection between
the unbordered strings of length $n$ and the prime palstars of length
$2n$.   As a consequence they obtained a formula for the number of prime
palstars.

Despite some interest in the palstars themselves
\cite{Manacher:1975,Galil&Seiferas:1978}, it seems no one has
enumerated them.  Here we observe that bijection mentioned previously,
together with the unique factorization of palstars, provides an
asymptotic enumeration for the number of palstars.

\section{Generating function for the palstars}

Again, let $k \geq 2$ denote the size of the alphabet.
Let $p_k(n)$ denote the number of palstars of length $2n$, and
let $u_k (n)$ denote the number of unbordered strings
of length $n$.

\begin{lemma}
For $n \geq 1$ and $k \geq 2$ we have
$$ p_k (n) = \sum_{1 \leq i \leq n} u_k(i) p_k(n-i) .$$
\label{one}
\end{lemma}

\begin{proof}
Consider a palstar of length $2n > 0$.  Either it is a prime palstar,
and by \cite{Rampersad&Shallit&Wang:2011}
there are $u_k(n) = u_k (n) p_k (0)$ of them,
or it is the concatenation of two or
more prime palstars.  In the latter case,
consider the length of this first factor;
it can potentially be $2i$ for $1 \leq i \leq n$.  Removing this first
factor, what is left is 
also a palstar.  This gives $u_k (i) p_k (n-i)$ distinct palstars
for each $i$.  Since factorization into prime palstars
is unique, the result follows.
\end{proof}

Now we define generating functions as follows:
\begin{eqnarray*}
P_k (X)  &=& \sum_{n \geq 0} p_k (n) X^n \\
U_k (X) &=& \sum_{n \geq 0} u_k (n) X^n .
\end{eqnarray*}

The first few terms are as follows:
\begin{eqnarray*}
P_k (X) &=&  1 + kX + (2k^2-k)X^2 + (4k^3-3k^2)X^3 + (8k^4-8k^3+k)X^4 + \cdots \\
U_k (X)  &=& 1 + kX + (k^2-k)X^2 + (k^3-k^2)X^3 + (k^4-k^3-k^2+k)X^4 + \cdots  .
\end{eqnarray*}

\begin{theorem}
$$P_k (X) = {1 \over {2-U_k (X)}}.$$
\label{two}
\end{theorem}

\begin{proof}
From Lemma~\ref{one}, we have
\begin{eqnarray*}
U_k (X) P_k (X) &=& \left(\sum_{n \geq 0} u_k (n) X^n \right) 
\left(\sum_{n \geq 0} p_k (n) X^n \right) \\
&=& 1 + \sum_{n \geq 1} \left(\sum_{0 \leq i \leq n} u_k(i) p_k (n-i) 
\right) X^n \\
&=& 1+ \left( \sum_{n \geq 1} \sum_{1 \leq i \leq n} u_k(i) p_k (n-i) X^n \right) 
	+ \sum_{n \geq 1} p_k(n) X^n \\
&=& 1 + \left( \sum_{n \geq 1} p_k (n) X^n \right) + \sum_{n \geq 1} p_k (n) X^n \\
&=& 2 P_k (X) - 1,
\end{eqnarray*}
from which the result follows immediately.
\end{proof}

\section{The main result}

\begin{theorem}
For all $k \geq 2$ there is a constant $\alpha_k$ with
$2k -1 < \alpha_k < 2k-{1 \over 2}$ such that
the number of palstars of length $2n$ is $\Theta(\alpha_k^n)$.
\end{theorem}

\begin{proof}
From Theorem~\ref{two} and the
``First Principle of Coefficient Asymptotics''
\cite[p.~260]{Flajolet&Sedgewick:2009}, it follows that
the asymptotic behavior of $[X^n] P_k (X)$, the coefficient of $X^n$ in
$P_k (X)$, is controlled by the behavior of the roots of $U_k (X) = 2$.
Since $u_{k}(0) = 1$ and $U_{k}(X) \rightarrow \infty$ as $X
\rightarrow \infty$, the equation $U_k (X) = 2$
has a single positive real root, which is 
$\rho = \rho_k = \alpha_k^{-1}$.  We first show that
$2k -1 < \alpha_k < 2k-{1 \over 2}$.

Recalling that $u_k(n)$ is the number of unbordered strings of length
$n$ over a $k$-letter alphabet, we see that $u_k(n) \leq k^n - k^{n-1}$
for $n \geq 2$, since $k^n$ counts the total number of strings of
length $n$, and $k^{n-1}$ counts the number of strings with a border
of length $1$.  Similarly
$$u_k(n) \geq \begin{cases}
	k^n - k^{n-1} - \cdots - k^{n/2}, & \text{if $n \geq 2$ is even}; \\
	k^n - k^{n-1} - \cdots - k^{(n+1)/2}, & \text{if $n \geq 2$ is odd},
	\end{cases} 
$$
since this quantity represents removing strings with borders of lengths
$1, 2, \ldots, n/2$ (resp., $1,2, \ldots, (n+1)/2$) if $n$ is even
(resp., odd) from the total number.  Here we use the classical fact
that if a word of length $n$ has a border, it has one of length $\leq n/2$.

It follows that for real $X > 0$ we have
\begin{align*}
U_k (X) &= \sum_{n \geq 0} u_k(n) X^n \\
&= 1 + kX + \sum_{n \geq 2} u_k(n) X^n \\
&\leq 1 + kX + \sum_{n \geq 2} (k^n - k^{n-1}) X^n \\
& = {{kX^2 - 1} \over {kX - 1 }} .
\end{align*}

Similarly for real $X > 0$ we have 
\begin{align*}
U_k(X) & = \sum_{n \geq 0} u_k(n) X^n \\
&= 1 +kX + \sum_{l \geq 1} u_k (2l) X^{2l} +
\sum_{m \geq 1} u_k (2m+1) X^{2m+1} \\
&\geq 1 + kX + \sum_{l \geq 1} (k^{2l} - k^{2l-1} - \cdots - k^l) X^{2l} 
+ \sum_{m \geq 1} (k^{2m+1} - \cdots - k^{m+1}) X^{2m+1} \\
&= {{1-2kX^2} \over {(kX-1)(kX^2-1)}} .
\end{align*}

This gives, for $k \geq 2$, that
$$ 2 < {{(2k-1) (4k^2 - 6k+1) } \over {(k-1)^2 (4k-1)}} \leq
U_k \left({1 \over {2k-1}} \right) $$
and
$$ U_k\left({ 1 \over {2k-{1\over 2}}}\right) \leq {{16k^2-12k+1} \over {(4k-1)(2k-1)}} < 2 .$$
It follows that
${1 \over {2k -{1 \over 2}}} \leq \rho_k \leq {1 \over {2k-1}} $ and hence
$2k -1 < \alpha_k < 2k-{1 \over 2}$.

To understand the asymptotic behavior of $[X^n] P_k (X)$, we need
to rule out other (complex) roots with the same absolute value as
$\rho$.
Suppose there is another solution
$X = \rho e^{i\psi}$ with  $-\pi < \psi < 0$ or $0 < \psi \le \pi $.
Then
$$2 = \sum_{n \ge 0}u_{k}(n)\rho^{n}e^{in\psi} = \sum_{n \ge 0}u_{k}(n)\rho^{n}\cos n\psi + i\sum_{n \ge 0}u_{k}(n)\rho^{n}\sin n\psi.$$
We must have $\sum_{n \ge 0}u_{k}(n)\rho^{n}\sin n\psi = 0$. Hence
$$2 = \sum_{n \ge 0}u_{k}(n)\rho^{n}\cos n\psi.$$
Now if $| \cos n\psi | < 1$ for some $n$ then
$$2 = \left|\sum_{n \ge 0}u_{k}(n)\rho^{n}\cos n\psi \right| \le \sum_{n \ge 0}u_{k}(n)\rho^{n}|\cos n\psi| < \sum_{n \ge 0}u_{k}(n)\rho^{n} = 2 .$$
This is a contradiction, so $|\cos n\psi| = 1$ for all $n$. 
Hence $\cos n\psi = \pm 1$ for all $n$. Thus $n\psi = \pm \pi + 2\pi l_{n}$ for all $n$ and $l_{n}$ is an integer for all $n$. Since $\cos x = 
\cos (-x)$, we may suppose that $0 < \psi \le \pi$. If $\psi = \pi$ then $X = -\rho$. 
But then, using the fact that $u_{k}(1) = k$, we get
$$2 = \sum_{n = 0}^{\infty}u_{k}(n)\rho^{n}(-1)^{n} < \sum_{n = 0}u_{k}(n)\rho^{n} = 2.$$
This contradiction shows that we may suppose $0 < \psi < \pi$.

Suppose $\cos n\psi = \pm 1$ for all $n$. Then
for all $n$
$$n\psi = \pm \pi + 2\pi l_{n}$$
where $l_{n}$ is an integer. Thus $n\psi/\pi = \pm 1 + 2l_{n}$ for all $n$. 
From Dirichlet's diophantine approximation theorem 
(e.g., \cite[Thm.~185]{Hardy&Wright:1971}),
given $\psi/\pi$ and an integer $q \ge 1$,
there are infinitely many $n$ and integers $x$ such that 
$$\left|n\frac{\psi}{\pi} - x\right| < \frac{1}{q}.$$
Thus $x - 1/q < n\psi/\pi < x + 1/q$ and 
$$\left|\pm 1 + 2l_{n} - x\right| < \frac{1}{q}.$$
Choosing $q > 1$ we see that $\pm 1 + 2l_{n} - x$ is an integer
$<1$ in absolute value.  Thus
$\pm 1 + 2l_{n} -x = 0$, and so $n\psi/\pi$ is an 
integer for infinitely many $n$.
Thus $\psi/\pi = p/m$ or $\psi = (p/m)\pi$ and we may suppose $p$ and $m$ are coprime. We have seen that if $X = \rho^{n} e^{nip/m}$ and if $|\cos np/m| < 1$ for any $n$ we have a contradiction. Therefore for all $n$ we have 
$$\cos\left(\frac{n p \pi}{m}\right) = \pm 1.$$
Thus $n p \pi/m = l\pi$ for all $n$. Thus $m$ divides $n$ for all $n$. This is a contradiction if $n = m + 1$. Thus
$$\frac{1}{2 - U_{k}(x)}$$
has only one singularity $x = \rho > 0$ with $|x| = \rho$.

It remains to determine the order of the zero $\rho$.
From above $U_{k}(X) = 2$ has a solution $\alpha_{k}^{-1}$ which
satisfies $2k - 1 < \alpha_{k} < 2k - {1 \over 2}$. 
Nielsen \cite{Nielsen:1973} showed that $u_k (n) \sim c_{k}k^{n}$ for
a constant $c_k$.
Thus $U_{k}(X)$ has radius of convergence $1/k$.
Thus $1/\alpha_{k}$ is in the region where $U_{k}$ is analytic.
Thus $2 - U_{k}(X)$ is analytic at $1/\alpha_{k}$ and has a zero at $1/\alpha_{k}$ of multiplicity $m$.
If $m \ge 2$ then the derivative of $2 - U_{k}(X)$ equals $0$
at $X = 1/\alpha_{k}$.
However $U_{k}^{'}(X) > 0$ since $u_{k}(n) > 0$ for some $n$.
Thus $2 - U_{k}(X)$ has a simple zero at $X = 1/\alpha_{k}$,
and so $P_{k}(X)$ has a simple pole at $X = 1/\alpha_{k}$.
Near $\alpha_{k}^{-1}$ the generating function $ U_{k}(X)$ has the expansion $2 + C_k(X - \alpha_{k}^{-1}) + C'(X - \alpha_{k}^{-1})^{2} + \cdots$ with
$C_k > 0$.
Furthermore
$$P_{k}(X) = \frac{1}{2 - U_{k}(X)} = \frac{1}{-C_k(X - 1/\alpha_{k})
 - C'(X - 1/\alpha_{k})^{2} + \cdots}.$$
Now
$$P_{k}^{'}(X) = \frac{U_{k}^{'}(X)}{(2 - U_{k}(X))^{2}},$$
so there is a positive $\delta$ such that 
$$[X^{n}]P_{k}(X) = [X^{n}]\frac{1}{C_k(1/\alpha_{k} -X)} +\cdots = [X^{n}]\frac{\alpha_{k}}{C_k}\left(\frac{1}{1 - \alpha_{k}X}\right) + \cdots
= \frac{\alpha_{k}^{n + 1}}{C_k} + O\left(\left(\alpha_{k} - \delta\right)^{n}\right),$$
since 
$$P_{k}(X) - \frac{\alpha_{k}}{C_k}\left(\frac{1}{1 - \alpha_{k}X}\right)$$
has no singularity on $|X| = 1/\alpha_{k}$ so has radius of convergence $> 1/\alpha_{k}$.
Here
$$C_k = U_{k}^{'}\left(\frac{1}{\alpha_{k}}\right).$$

It now follows from standard results 
(e.g., \cite[Thm.~IV.7, p.~244]{Flajolet&Sedgewick:2009}) that 
$$ [X^n] P_k (X) = {{\alpha_k^{n+1}} \over {C_k}} + O((\alpha_k - \delta)^n)
	= \Theta(\alpha_k^n) .$$
\end{proof}

\section{Numerical results}

Here is a table giving the first few values of $P_k (n)$.

\begin{figure}[H]
\begin{center}
\begin{tabular}{cccccccccccc}
\hline
$n = $ & 0 & 1 & 2 & 3 & 4 & 5 & 6 & 7 & 8 & 9 & 10  \\
\hline
$k = 2$& 1 & 2 & 6 & 20 & 66 & 220 & 732 & 2440 & 8134 & 27124 & 90452  \\
$k = 3$& 1 & 3 & 15 & 81 & 435 & 2349 & 12681 & 68499 & 370023 & 1998945 & 10798821 \\
$k = 4$& 1 & 4 & 28 & 208 & 1540 & 11440 & 84976 & 631360 & 4690972 & 34854352 & 258971536 \\
\hline
\end{tabular}
\end{center}
\end{figure}

By truncating the power series $U_k(X)$ and solving the equation
$U_k (X) = 2$ we get better and better approximations to
$\alpha_k^{-1}$.  For example, for $k = 2$ we have
\begin{eqnarray*}
\alpha_2^{-1} &\doteq & 0.29983821359352690506155111814579603919303182364781730366339199333065202 \\
\alpha_2 & \doteq & 3.3351319300335793676678962610376244842363270634405611577104447308511860 \\
C_2 & \doteq & 6.278652437421018217684895562492005276088368718322063642652328654828673
\end{eqnarray*}

To determine an asymptotic expansion for $\alpha_k$ as $k \rightarrow \infty$,
we compute the Taylor series expansion for
$P_k(n)/P_k(n+1)$, treating $k$ as an indeterminate,
for $n$ large enough to cover the error
term desired.  For example, for $O(k^{-10})$ it suffices to take
$k = 16$, which gives
$$
\alpha_k^{-1} = {1 \over {2k}} + {1 \over {8k^2}} + {3 \over {32k^3}} + {1 \over {16k^4}} + {{27} \over {512k^5}} +  {{93} \over {2048k^6}} + 
{{83} \over {2048k^7}} + {{155} \over {4096k^8}} + {{4735} \over {131072k^9}} +
O(k^{-10}) 
$$
and hence
$$ \alpha_k = 2k - {1 \over 2} - {1 \over {4k}} - {3 \over {32k^2}} - {5 \over {64k^3}} - {{31} \over {512k^4}} - {{25} \over {512k^5}}  - {{23} \over {512 k^6}} - {{683} \over {16384k^7}} + O(k^{-8}).$$

\end{document}